\documentclass{jcmlatex}
\usepackage{subfigure}
\usepackage{graphicx}
\usepackage{float}
\usepackage{epstopdf}
\usepackage{cases}
\usepackage{empheq}
\makeatletter 
\@addtoreset{equation}{section}
\makeatother  

\def\JCMvol{xx}
\def\JCMno{x}
\def\JCMyear{20xx}
\def\JCMreceived{xxx}
\def\JCMrevised{xxx}
\def\JCMaccepted{xxx}

\begin{document}

\markboth{N.~LI AND Y.C.~ZHAO}{Convergence Analysis for the Cahn-Hilliard Equation}

\title{Convergence Analysis of a Linear, Unconditionally Energy-Stable SAV Finite Element Method for the Cahn-Hilliard Equation}

\author{Na LI
\thanks{School of Data Science and Computing, Shandong Women's University, Jinan, China \\ Email: dajiena2002@163.com}
\and
Yongchao ZHAO
\thanks{Jinan Quanxin Eco-Environmental Technology Co., Ltd., Jinan, China\\ Email: zhao538@163.com}}

\maketitle
\makeatletter
\let\ps@firstpage@orig\ps@firstpage

\renewcommand\ps@firstpage{%
  \ps@firstpage@orig%

  \def\@oddfoot{%
    \raisebox{-10pt}[0pt][0pt]{
      \parbox{\textwidth}{%
        \raggedright\footnotesize%
        \quad This work was supported by the Natural Science Foundation of Shandong Province, China (Grant No. ZR2023QA018).%
      }%
    }%
  }%
  \let\@evenfoot\@oddfoot%
}
\makeatother

\begin{abstract}
This paper proposes a finite element scheme, based on the Scalar Auxiliary Variable (SAV) approach, for the Cahn-Hilliard equation -- a model that possesses significant physical relevance and a rich mathematical structure. A convergence analysis of the fully discrete scheme is conducted under suitable regularity assumptions, confirming optimal-order convergence in both time and space for the phase variable, chemical potential, and auxiliary variable in the $H^1$-norm. Furthermore, the scheme is proven to be unconditionally energy stable. Finally, a numerical example is presented to demonstrate the effectiveness of the method and to confirm the theoretical convergence rates.
\end{abstract}

\begin{classification}
65M06, 65Z05.
\end{classification}

\begin{keywords}
Cahn-Hilliard equation, Scalar Auxiliary Variable(SAV), finite element method, error estimates.
\end{keywords}

\section{Introduction}
\label{sec:into}
The phase-field method is one of the primary approaches for simulating multiphase flow problems, such as phase separation in alloys, polymer blending, and biofilm dynamics (e.g., \cite{Linzonghu}). The core idea of this approach is to replace the theoretically sharp interface of zero thickness with a diffuse interface of finite thickness. Within this transition layer, physical quantities vary continuously, thereby circumventing the complex topological tracking required by traditional sharp-interface methods. The Cahn-Hilliard (CH) equation, as a classical phase-field model describing phase separation and coarsening dynamics, possesses profound physical background and rich mathematical structure. Its numerical solution has long been an important topic in computational mathematics and engineering simulation.

The CH equation is a fourth-order nonlinear diffusion equation, and its numerical solution faces two major challenges: first, the nonlinear double-well potential term introduces numerical stiffness, imposing strict restrictions on the time step size; second, the high-order derivative terms require sufficient regularity in spatial discretization. Traditional fully implicit schemes, while stable, require solving nonlinear systems at each step, leading to high computational costs. On the other hand, explicit or semi-explicit treatments simplify computation but often sacrifice energy stability or impose severe time-step constraints. Consequently, the central goal of numerical method development for the CH equation is to construct schemes that are both discretely energy stable (ensuring the physical reliability of long-time simulations) and computationally efficient (allowing for the use of relatively large time steps).

To preserve the energy dissipation structure of the original system at the discrete level, various energy-stable numerical schemes have been proposed, including the convex splitting approach (e.g., \cite{C.M.Elliott,D.J.Eyre}), the stabilized explicit method (e.g., \cite{J.Zhu, J.Shen4, P.Yue, D.Li}), and the Invariant Energy Quadratization (IEQ) method (e.g., \cite{X.Yang,X.Yang1}). Among these, the Scalar Auxiliary Variable (SAV) method (e.g., \cite{J.Shen2,J.Shen5}), introduced by Shen et al. in 2018, has garnered significant attention. This approach rewrites the nonlinear part of the free energy as a quadratic form by introducing an auxiliary variable. This reformulation leads to a numerical scheme that is linear, unconditionally energy stable, and readily amenable to high-order temporal discretizations. The SAV method inherits the advantages of the IEQ method but, crucially, only requires solving constant-coefficient linear systems at each time step, which significantly boosts computational efficiency. Furthermore, it is independent of the specific form of the nonlinear potential, making it applicable to a broad class of gradient flow systems.

Since its inception, the SAV method has been the subject of extensive research, leading to numerous improvements and generalizations. For instance, the Multiple SAV (MSAV) method enhances accuracy and stability by using multiple auxiliary variables (e.g., \cite{Q.Cheng3}); the relaxed SAV (RSAV) method controls numerical errors in the auxiliary variables via a relaxation technique (e.g., \cite{M.Jiang}); the Runge-Kutta based RK-SAV method achieves arbitrary high-order accuracy (e.g., \cite{G.Akrivis}); the generalized Positive Auxiliary Variable (gPAV) method handles more general dissipative systems (e.g., \cite{Y.Zhiguo}); and the Lagrange multiplier-based SAV method preserves the original energy functional exactly at the discrete level (e.g., \cite{Q.Cheng4}). These variants have successfully extended the application of SAV to a wide range of phase-field models, including those coupled with fluid flow, crystal growth, and vesicle membranes.

However, within the finite element framework, a complete and unified a priori error analysis theory for the Scalar Auxiliary Variable (SAV) method applied to the classical Cahn-Hilliard equation remains to be established. Although convergence proofs have been provided in other discretization frameworks such as spectral methods (e.g., \cite{Zhaoyang,Linping}), the systematic development of an optimal-order convergence theory in both time and space for its fully discrete backward Euler scheme in the finite element setting remains essential to strengthening the numerical analysis foundation of this method.

In this paper, we propose a fully discrete finite element scheme based on the SAV approach for the classical Cahn-Hilliard equation, employing backward Euler temporal discretization and linear finite element spatial discretization. A systematic convergence analysis of the proposed scheme is conducted: we first prove its unconditional energy stability and unique solvability. Then, by introducing an appropriate elliptic projection operator, we derive rigorous error estimates for the phase-field variable, the chemical potential, and the auxiliary scalar variable, thereby establishing optimal first-order convergence in time and optimal-order convergence in space. Finally, a numerical experiment is presented to verify the effectiveness of the scheme and to confirm that the observed convergence rates align with the theoretical predictions.

The remainder of this paper is organized as follows. In Section \ref{sec:formulation}, we present the SAV reformulation of the CH equation, its weak formulation, and the finite element space settings. In Section \ref{sec:fully discrete}, we introduce the fully discrete SAV scheme, prove its energy stability and unique solvability, and perform a detailed error analysis using a projection operator to establish our main convergence theorem. Section \ref{sec:numerical} provides numerical examples to verify the stability and convergence order of the scheme.

\section{Mathematical Formulation and Preliminaries}
\label{sec:formulation}

In this paper, we consider the classical Cahn-Hilliard equation with the boundary and initial conditions:
\begin{equation}\label{model}
  \left\{
   \begin{aligned}
 & \phi_{t}=\Delta \mu, && \mathrm{in}\quad \Omega \times J, \\
 & \mu=-\Delta \phi+F'(\phi), && \mathrm{in}\quad\Omega \times J,\\
 & \partial_{\mathbf{n}}\phi\mid_{\Gamma}=0, \partial_{\mathbf{n}}\mu\mid_{\Gamma}=0, && \mathrm{on}\quad\Gamma \times J,\\
 & \phi\mid_{t=0}=\phi_{0}, && \mathrm{in}\quad\Omega,
   \end{aligned}
   \right.
  \end{equation}
where $\Omega$ is a bounded domain with smooth boundary $\Gamma$ in $\mathbb{R}^{d} (d=2,3)
$, $J=(0,T]$ denotes the time interval with $T>0$ being the total simulation time. The subscript $t$ denotes the partial derivative with respect to time and $\Delta$  represents the usual Laplace operator acting on the space variables in $\Omega$, $\mathbf{n}=\mathbf{n}(x)$ is the unit outer normal vector on $\Gamma$, $\partial_{\mathbf{n}}$ denotes the outward normal derivative. The phase variable $\phi$ stands for the difference of the local relative concentration of the two components in the binary mixture such that $\phi=\pm 1$ correspond to the pure phases components of the mixture, while $\phi \in (-1,1)$ corresponds to the transition between them. $\mu$ represents the chemical potential that equals to the $Fr\acute{e}chet$ derivative of certain free energy given by
$$E_{\mathrm{total}}(\phi)=\int_{\Omega}\left(\frac{1}{2}|\nabla\phi|^{2}+F(\phi)\right)\mathrm{d}x.$$
$F$ denotes the potential, which typically has a double-well structure with two minimums and a local unstable maximum in between. The total free energy decreases over time
\begin{equation}
\frac{\mathrm{d}}{\mathrm{d}t}E_{\mathrm{total}}(\phi)+\int_{\Omega} |\nabla \mu|^{2}\mathrm{d}x = 0, \quad \forall t\in [0,T].\label{energysh}
\end{equation}
Assume that the potential energy $E(\phi):=\int_{\Omega} F(\phi)  \mathrm{d}x$ is bounded from below, i.e., $E(\phi) > -C_0$ for some constant $C_0 > 0$. In the SAV approach, we introduce a scalar auxiliary variable
$$ r(t) = \sqrt{E(\phi) + C_0}. $$

$\mathbf{Notation\ Convention:}$ For the remainder of this paper, we adopt the convention that $E(\phi)$ denotes $E(\phi)+C_0$ (not the original energy functional). This redefinition ensures $E(\phi) > 0$ for all $\phi$, which guarantees the well-definedness of terms involving $\sqrt{E(\phi)}$. This convention does not alter the gradient flow structure of the system since $C_0$ is a constant.

The original system can then be reformulated as:
\begin{empheq}[left=\empheqlbrace]{alignat=2}
    &\phi_{t}   = \Delta \mu,                                           &&\quad \text{in } \Omega \times J, \label{1} \\
    &\mu        = -\Delta \phi + \frac{r}{\sqrt{E(\phi)}} F'(\phi),     &&\quad \text{in } \Omega \times J, \label{2} \\
    &r_{t}      = \frac{1}{2\sqrt{E(\phi)}} \int_{\Omega} F'(\phi) \phi_{t} \, \mathrm{d}x, &&\quad \text{in } \Omega \times J, \label{3} \\
    &\partial_{\mathbf{n}}\phi \mid_{\Gamma} = 0,\quad \partial_{\mathbf{n}}\mu \mid_{\Gamma} = 0, &&\quad \text{on } \Gamma \times J, \label{4} \\
    &\phi|_{t=0} = \phi_{0}(x),                                         &&\quad \text{in } \Omega. \label{5}
\end{empheq}

We give some preliminaries,
Let
$$(u,v) = (u,v)_{L^2(\Omega)} = \int_{\Omega} uv \, dx$$
denote the $L^2(\Omega)$ inner product. Let $W^{k,p}(\Omega)$ be the standard Sobolev space
$$W^{k,p}(\Omega) = \{v : \|v\|_{W^{k,p}(\Omega)} < \infty\},$$
where
$$\|v\|_{W^{k,p}(\Omega)} = \left(\sum_{|\alpha|\leq k} \|D^{\alpha}v\|_{L^p(\Omega)}^p\right)^{1/p}.$$

For notational convenience, we define the following shorthand notations:
$$H^{k}(\Omega):= W^{k,2}(\Omega),\ L^{2}(\Omega):= H^{0}(\Omega),$$
$$\|v\|_{k,p}:= \|v\|_{W^{k,p}(\Omega)},\ \|v\|_k:= \|v\|_{H^k(\Omega)} = \|v\|_{k,2},\ \|v\|:= \|v\|_{L^2(\Omega)} = \|v\|_0.$$

To analyze the nonlinear part of the equation in the subsequent discussion, we introduce the negative norm:
$$\|v\|_{-1}=\mathrm{sup}\left\{\frac{(v,\zeta)}{\|\zeta\|_{1}}, \ \forall\zeta \in H^{1}(\Omega)\right\}.$$

Multiplying \eqref{1} with a test function $u$, then using integration by parts,
 $$(\phi_{t},u)=(\Delta \mu,u)=\int_{\Omega}\Delta \mu \cdot u \mathrm{d}x=\int_{\Gamma}\partial_{\mathbf{n}}\mu\cdot u \mathrm{d}s-\int_{\Omega}\nabla \mu\cdot\nabla u \mathrm{d}x=-(\nabla \mu,\nabla u ).$$
Similarly, multiplying \eqref{2} with a test function $v$,
\begin{align*}
 & (\mu,v)=(\nabla \phi,\nabla v)+ \left(\frac{r}{\sqrt{E(\phi)}}F'(\phi),v\right).
  \end{align*}

Let $V = H^1(\Omega)$ denote the standard Sobolev space. The weak formulation of the system \eqref{1}-\eqref{5} is: find $\{ \phi, \mu, r \} \in V \times V \times \mathbb{R}$ (with the additional regularity $\phi, \mu \in H^{2}(\Omega) \cap L^{\infty}(\Omega)$ for the purpose of error analysis) such that
\begin{empheq}[left=\empheqlbrace]{alignat=2}
    &(\phi_{t}, u) = -(\nabla \mu, \nabla u), &&\qquad \forall u \in V, \label{r1} \\
    &(\mu, v)      = (\nabla \phi, \nabla v) + \left(\frac{r}{\sqrt{E(\phi)}} F'(\phi), v\right), &&\qquad \forall v \in V, \label{r2} \\
    &r_{t}         = \left(\frac{1}{2\sqrt{E(\phi)}} F'(\phi), \phi_{t}\right). &&\qquad \label{r3}
\end{empheq}

To illustrate an appropriate finite element approximation for $\{\phi, \mu, r\}$, we divide the space $\Omega$ into triangular elements $\tau_{h}$ quasi-uniformly, where the diameter of $\tau_{h}$ is not greater than $h(0<h<1)$, and each boundary triangular element is allowed to have a curved edge. We construct a conforming finite element space $V_h \subset V$ as follows:
$$V_h = \{ v_h \in C(\overline{\Omega}) : v_h|_e \in P_1(e), \ \forall e \in \tau_h \},$$
where $P_{q}(e)$ represents the space of polynomials of degree less than or equal to $q$ on cell $e$.

Let $\pi_h: V \rightarrow V_h$ denote the standard $L^2$-orthogonal projection operator, defined by
$$(v - \pi_h v, u_h) = 0, \quad \forall u_h \in V_h. $$
Under the regularity assumption $v \in H^2(\Omega)$, the following approximation properties hold:
\begin{align}
\| \pi_h v \|_1 &\leq C \| v \|_1, \label{ty1} \\
\| v - \pi_h v \| + h \| v - \pi_h v \|_1 &\leq C h^2 \| v \|_2. \label{ty2}
\end{align}
Here the letter $C$ denotes a positive constant independent of the discretization parameters $h$ and $\Delta t$, and may have different values at different places.

\section{The Fully Discrete Scheme and Error Estimates}
\label{sec:fully discrete}
In this section, we will present a fully discrete scheme for the system, show the discrete energy stability, and analyze error estimates for the phase variable, the chemical potential, and the auxiliary scalar variable.

\subsection{ Existence and Uniqueness of Solution and Unconditional Energy Stability}
Based on the backward Euler temporal discretization and the Scalar Auxiliary Variable (SAV) approach, the fully discrete scheme of the system is as follows:
\begin{empheq}[left=\empheqlbrace]{alignat=2}
    &(\partial_{t}\phi_{h}^{n+1}, u_{h}) = -(\nabla \mu_{h}^{n+1}, \nabla u_{h}), &&\qquad \forall u_{h} \in V_h, \label{f1} \\
    &(\mu_{h}^{n+1}, v_{h}) = (\nabla \phi_{h}^{n+1}, \nabla v_{h}) + \left(r_{h}^{n+1}\frac{F'(\phi_{h}^{n})}{\sqrt{E(\phi_{h}^{n})}}, v_{h}\right), &&\qquad \forall v_{h} \in V_h, \label{f2} \\
    &\partial_{t}r_{h}^{n+1} = \left(\frac{F'(\phi_{h}^{n})}{2\sqrt{E(\phi_{h}^{n})}}, \partial_{t}\phi_{h}^{n+1}\right), && \label{f3}
\end{empheq}
 where $(\phi_{h}^{n}, \mu_{h}^{n},r_{h}^{n})\in V_h\times V_h\times \mathbb{R}$ is an approximation of $\phi(t^{n})=\phi(n \Delta t)$, $\mu(t^{n})=\mu(n \Delta t)$ and $r(t^{n})=r(n \Delta t)$, $\partial_{t}\phi_{h}^{n+1}=\frac{\phi_{h}^{n+1}-\phi_{h}^{n}}{\Delta t}$, $\partial_{t}r_{h}^{n+1}=\frac{r_{h}^{n+1}-r_{h}^{n}}{\Delta t}$ ($\Delta t>0$ stands for the time step).

 \begin{theorem}\label{thm:energy_stability}
The system~\eqref{f1}--\eqref{f3} is unconditionally uniquely solvable and is unconditionally stable with respect to the discrete total free energy
\begin{equation}
E_{\mathrm{total}}^{n}=\frac{1}{2}\|\nabla\phi_{h}^{n}\|^{2}+(r_{h}^{n})^{2}. \label{eq:total_energy}
\end{equation}
Specifically, the following energy inequality holds:
\begin{align}
E_{\mathrm{total}}^{n+1}-E_{\mathrm{total}}^{n}
+\frac{1}{2}\|\nabla(\phi_{h}^{n+1}-\phi_{h}^{n})\|^{2}
+(r_{h}^{n+1}-r_{h}^{n})^{2}
=-\|\nabla\mu_{h}^{n+1}\|^{2}\Delta t, \label{rff1}
\end{align}
which demonstrates the monotonic decay of the discrete energy.
\end{theorem}

 Proof. First, we establish the unconditional unique solvability of the numerical scheme.

  Since \eqref{f1}-\eqref{f3} can be rewritten as the following linear system \eqref{weiyi} for unknowns $\phi_{h}^{n+1}$, $\mu_{h}^{n+1}$ and $r_{ h}^{n+1}$,
\begin{equation}\label{weiyi}
  \left\{
   \begin{aligned}
&(\phi_{h}^{n+1},u_{h})+\Delta t (\nabla \mu_{h}^{n+1},\nabla u_{h})=(\phi_{h}^{n}, u_{h}),\\
&(\mu_{h}^{n+1},v_{h})-(\nabla \phi_{h}^{n+1},\nabla v_{h})-\left(r_{h}^{n+1}\frac{F'(\phi_{h}^{n})}{\sqrt{E(\phi_{h}^{n})}},v_{h}\right)=0,\\
&r_{h}^{n+1}-\left(\frac{F'(\phi_{h}^{n})}{2\sqrt{E(\phi_{h}^{n})}}, \phi_{h}^{n+1}\right)= r_{h}^{n}-\left(\frac{F'(\phi_{h}^{n})}{2\sqrt{E(\phi_{h}^{n})}}, \phi_{h}^{n}\right),
   \end{aligned}
   \right.
  \end{equation}
in order to explain the existence and uniqueness of the solution of equation \eqref{f1}-\eqref{f3}, it suffices to show the corresponding homogeneous system of linear equations
\begin{numcases}{ }
(\phi_{h}^{n+1},u_{h})+\Delta t (\nabla \mu_{h}^{n+1},\nabla u_{h})=0,\label{qc1}\\
(\mu_{h}^{n+1},v_{h})-(\nabla \phi_{h}^{n+1},\nabla v_{h})-\left(r_{h}^{n+1}\frac{F'(\phi_{h}^{n})}{\sqrt{E(\phi_{h}^{n})}},v_{h}\right)=0,\label{qc2}\\
r_{h}^{n+1}-\left(\frac{F'(\phi_{h}^{n})}{2\sqrt{E(\phi_{h}^{n})}}, \phi_{h}^{n+1}\right)=0.\label{qc3}
\end{numcases}
has only the trivial solution.

Take $u_{h}=\phi_{h}^{n+1}$ and $u_{h}=\mu_{h}^{n+1}$ in \eqref{qc1}, respectively,
\begin{align}
\|\phi_{h}^{n+1}\|^{2}+\Delta t (\nabla \mu_{h}^{n+1},\nabla\phi_{h}^{n+1})=0,\label{phi21}\\
(\phi_{h}^{n+1},\mu_{h}^{n+1})+\Delta t \|\nabla \mu_{h}^{n+1}\|^{2}=0.\label{gphi1}
\end{align}

Take $v_{h}=\phi_{h}^{n+1}$ and $v_{h}=\mu_{h}^{n+1}$ in \eqref{qc2}, respectively,
\begin{align}
(\mu_{h}^{n+1},\phi_{h}^{n+1})-\|\nabla \phi_{h}^{n+1}\|^{2}-\left(r_{h}^{n+1}\frac{F'(\phi_{h}^{n})}{\sqrt{E(\phi_{h}^{n})}},\phi_{h}^{n+1}\right)=0,\label{gmu1}\\
\|\mu_{h}^{n+1}\|^{2}-(\nabla \phi_{h}^{n+1},\nabla \mu_{h}^{n+1})-\left(r_{h}^{n+1}\frac{F'(\phi_{h}^{n})}{\sqrt{E(\phi_{h}^{n})}},\mu_{h}^{n+1}\right)=0.\label{gmu2}
\end{align}

Multiply equation (\eqref{qc3} by $2 r_{h}^{n+1}$,
\begin{align}
2 (r_{h}^{n+1})^2-\left(r_{h}^{n+1} \frac{F'(\phi_{h}^{n})}{\sqrt{E(\phi_{h}^{n})}}, \phi_{h}^{n+1}\right)=0.\label{rr}
\end{align}

From \eqref{gphi1}, \eqref{gmu1}, and \eqref{rr}, we obtain
\begin{align}
\|\nabla \phi_{h}^{n+1}\|^{2}+\Delta t \|\nabla \mu_{h}^{n+1}\|^{2}+2 (r_{h}^{n+1})^2=0,\label{pm1}
\end{align}
so that
\begin{align}
\nabla \phi_{h}^{n+1}=0,  \nabla \mu_{h}^{n+1}=0, r_{h}^{n+1}=0. \label{pm2}
\end{align}

Considering \eqref{phi21}, \eqref{gmu2} and \eqref{pm2}, we get
$\|\phi_{h}^{n+1}\|=0, \|\mu_{h}^{n+1}\|=0,$ which implies $\phi_{h}^{n+1}=0, \mu_{h}^{n+1}=0,$
Therefore, equations \eqref{qc1}-\eqref{qc3} have only the trivial solution, and consequently, the solution of \eqref{f1}-\eqref{f3} exists and is unique.

Next, we analyze the energy stability properties of the scheme.

Take $u_{h}=\mu_{h}^{n+1}$ in \eqref{f1}, $v_{h}=\partial_{t}\phi_{h}^{n+1}$ in \eqref{f2}:
\begin{align}
(\partial_{t}\phi_{h}^{n+1},\mu_{h}^{n+1}) &= -(\nabla \mu_{h}^{n+1},\nabla \mu_{h}^{n+1}) = -\|\nabla\mu_{h}^{n+1}\|^{2}, \label{u2} \\
(\mu_{h}^{n+1},\partial_{t}\phi_{h}^{n+1}) &= (\nabla \phi_{h}^{n+1},\partial_{t}\nabla\phi_{h}^{n+1})
    + \left(r_{h}^{n+1}\frac{F'(\phi_{h}^{n})}{\sqrt{E(\phi_{h}^{n})}},\partial_{t}\phi_{h}^{n+1}\right) \nonumber \\
    &= \frac{1}{2\Delta t}\bigl(\|\nabla \phi_{h}^{n+1}\|^{2} - \|\nabla \phi_{h}^{n}\|^{2}
        + \|\nabla (\phi_{h}^{n+1} - \phi_{h}^{n})\|^{2}\bigr) \nonumber \\
    &\quad + \left(r_{h}^{n+1}\frac{F'(\phi_{h}^{n})}{\sqrt{E(\phi_{h}^{n})}},\partial_{t}\phi_{h}^{n+1}\right). \label{v2}
\end{align}

Let \eqref{f3} be multiplied by $2r_{h}^{n+1}$,
\begin{align*}
\partial_{t}r_{h}^{n+1} \cdot 2r_{h}^{n+1}=\left(r_{h}^{n+1}\frac{F'(\phi_{h}^{n})}{\sqrt{E(\phi_{h}^{n})}},\partial_{t}\phi_{h}^{n+1}\right),
\end{align*}
with that
\begin{align}
 \frac{1}{\Delta t}\left[(r_{h}^{n+1})^{2}-(r_{h}^{n})^2+(r_{h}^{n+1}-r_{h}^{n})^2\right]=\left(r_{h}^{n+1}\frac{F'(\phi_{h}^{n})}{\sqrt{E(\phi_{h}^{n})}},\partial_{t}\phi_{h}^{n+1}\right),\label{rf2}
\end{align}

Substituting \eqref{u2},\eqref{rf2} into \eqref{v2} and simplifying, we obtain \eqref{rff1}. Thus, Theorem~\ref{thm:energy_stability} is proved.$\qquad\Box$
\subsection{Introduction of a Projection}
In order to analyze the convergence order of the fully discrete scheme, the following projection is defined in this section: find $(\overline{\phi}, \overline{\mu}) \in V_h \times V_h$ such that
\begin{empheq}[left=\empheqlbrace]{alignat=2}
    &(\phi - \overline{\phi}, u_h) = 0, &&\qquad \forall u_h \in V_h, \label{t1} \\
    &(\mu - \overline{\mu}, v_h) = (\nabla(\phi - \overline{\phi}), \nabla v_h)
        + \Bigl(r \frac{F'(\phi)}{\sqrt{E(\phi)}} - r \frac{F'(\overline{\phi})}{\sqrt{E(\overline{\phi})}}, v_h\Bigr),
        &&\qquad \forall v_h \in V_h, \label{t2} \\
   &(\nabla(\mu - \overline{\mu}), \nabla w_h) = 0, &&\qquad \forall w_h \in V_h. \label{t3}
\end{empheq}

\begin{theorem}\label{thm:projection_error}
Assume that the solution $(\phi, \mu, r)$ of the weak problem \eqref{r1}-\eqref{r3} possesses the regularity $\phi, \mu \in H^{2}(\Omega) \cap L^{\infty}(\Omega)$, $r \in L^{\infty}(J)$ and $F \in C^{2}(\mathbb{R})$ with $F''$ bounded on bounded sets. Then, the projection problem \eqref{t1}-\eqref{t3} admits a unique solution $(\overline{\phi}, \overline{\mu}) \in V_h \times V_h$. Moreover, there exists a constant $C > 0$, independent of the mesh size $h$, such that the following optimal-order error estimate holds:
\begin{align}
\|\phi - \overline{\phi}\|_{1}^{2} + \|\mu - \overline{\mu}\|_{1}^{2} \leq C (\|\phi\|_{2}^{2} + \|\mu\|_{2}^{2}) h^{2}. \label{tywc}
\end{align}
\end{theorem}

\begin{proof} We first establish the existence and uniqueness of the projection $(\overline{\phi}, \overline{\mu})$. We observe that equation \eqref{t1} defines $\overline{\phi}$ as the standard $L^2$-orthogonal projection of $\phi$ onto the finite element space $V_h$, denoted by $\overline{\phi} = \pi_h \phi$. This projection is uniquely defined for any $\phi \in L^2(\Omega)$. Under the regularity assumption $\phi \in H^2(\Omega)$, the projection error satisfies the approximation property \eqref{ty2}, which yields
\begin{equation}
\|\overline{\phi}\| \leq C\|\phi\|, \quad \|\phi - \overline{\phi}\| \leq C\|\phi\|_2 h^2, \quad \|\phi - \overline{\phi}\|_1 \leq C\|\phi\|_2 h. \label{eq:rho_phi_est}
\end{equation}

Having established the properties of $\overline{\phi}$, we now turn to the existence and uniqueness of $\overline{\mu}$. With $\overline{\phi}$ fixed in $V_h$, we consider the system \eqref{t2}-\eqref{t3} for $\overline{\mu}$. To this end, we define a bilinear form $a: V_h \times V_h \to \mathbb{R}$ and a linear functional $F: V_h \to \mathbb{R}$ by
\begin{align*}
&a(\psi_h, v_h) = (\psi_h, v_h) + (\nabla \psi_h, \nabla v_h), \\
&F(v_h)= \left(\nabla(\phi - \overline{\phi}), \nabla v_h\right) + \left( r \frac{F'(\phi)}{\sqrt{E(\phi)}}, v_h \right) - \left( r \frac{F'(\overline{\phi})}{\sqrt{E(\overline{\phi})}}, v_h \right).
\end{align*}
The system \eqref{t2}-\eqref{t3} can then be reformulated as the variational problem: find $\overline{\mu} \in V_h$ such that
\begin{equation}
a(\overline{\mu}, v_h) = F(v_h), \quad \forall v_h \in V_h. \label{eq:mu_variational}
\end{equation}
Since $a(\cdot,\cdot)$ is continuous and coercive on $V_h$ with respect to the $H^1$-norm, and $F$ is bounded under the given regularity assumptions and the Lipschitz continuity of the nonlinear term, the Lax-Milgram theorem guarantees the existence of a unique solution $\overline{\mu} \in V_h$.

We now proceed to analyze the error $\mu - \overline{\mu}$.

Let $w_{h}=\pi_{h}\mu-\overline{\mu}$ in \eqref{t3},
\begin{align}
\|\nabla(\pi_{h}\mu-\overline{\mu})\|^{2}=-(\nabla(\mu-\pi_{h}\mu),\nabla(\pi_{h}\mu-\overline{\mu})),\label{mu31}
\end{align}
with that
\begin{align}
\|\nabla(\pi_{h}\mu-\overline{\mu})\|\leq\|\nabla(\mu-\pi_{h}\mu)\|\leq C \|\mu\|_{2} h. \label{mu32}
\end{align}

Take $v_{h}=\pi_{h}\mu-\overline{\mu}$ in \eqref{t2},
\begin{align}\label{mu21}
\|\pi_{h}\mu-\overline{\mu}\|^{2}=&-(\mu-\pi_{h}\mu,\pi_{h}\mu-\overline{\mu})+(\nabla(\phi-\overline{\phi}),\nabla(\pi_{h}\mu-\overline{\mu}))+\nonumber\\ &\left(r\frac{F'(\phi)}{\sqrt{E(\phi)}}-r\frac{F'(\overline{\phi})}{\sqrt{E(\overline{\phi})}},\pi_{h}\mu-\overline{\mu}\right)\nonumber\\
\leq & C \|\mu-\pi_{h}\mu\|^2+ \frac{1}{2}\|\pi_{h}\mu-\overline{\mu}\|^2+C(\|\nabla(\phi-\overline{\phi})\|^2+\|\nabla(\pi_{h}\mu-\overline{\mu})\|^2)\nonumber\\
&+C \left\|r\frac{F'(\phi)}{\sqrt{E(\phi)}}-r\frac{F'(\overline{\phi})}{\sqrt{E(\overline{\phi})}}\right\|^2.
\end{align}

Observing that
\begin{align}
\left\|r\frac{F'(\phi)}{\sqrt{E(\phi)}}-r\frac{F'(\overline{\phi})}{\sqrt{E(\overline{\phi})}}\right\|=&\left\|r\frac{F'(\phi)-F'(\overline{\phi})}{\sqrt{E(\phi)}}+r\left( \frac{ F'(\overline{\phi})}{\sqrt{E(\phi)}}-\frac{F'(\overline{\phi})}{\sqrt{E(\overline{\phi})}}\right)\right\|\nonumber\\
\leq&\left\|r \frac{ F'(\phi)-F'(\overline{\phi})}{\sqrt{E(\phi)}}\right\|+\left\|\frac{r F'(\overline{\phi})\left(E(\overline{\phi})-E(\phi)\right)}{\sqrt{E(\phi)E(\overline{\phi})}(\sqrt{E(\phi)}+\sqrt{E(\overline{\phi})})}\right\|\nonumber\\
\leq&C \|r\|_{L^{\infty}(J)}\|\phi-\overline{\phi}\|\leq C \|\phi\|_{2} h^{2},\label{F}
\end{align}
By combining the estimates from \eqref{eq:rho_phi_est}, \eqref{mu32} and \eqref{F}, we can obtain
\begin{align}\label{muty1}
\|\pi_{h}\mu-\overline{\mu}\|^{2} \leq & C (\|\phi\|_{2}^{2} h^{2}+\|\mu\|_{2}^2 h^{2}+ \|\phi\|_{2}^2+\|\mu\|_{2}^2)h^{2}\leq C(\|\phi\|_{2}^{2}+\|\mu\|_{2}^2) h^{2}.
\end{align}

Putting together the estimates for $\phi-\overline{\phi}$ and $\mu-\pi_{h}\mu$, we arrive at the desired error bound
$$\|\phi - \overline{\phi}\|_{1}^{2} + \|\mu - \overline{\mu}\|_{1}^{2} \leq C (\|\phi\|_{2}^{2} + \|\mu\|_{2}^{2})h^{2},$$
which completes the proof of the theorem.$\qquad \Box$
\end{proof}
\subsection{Error Estimates}
We estimate the error of numerical scheme \eqref{f1}-\eqref{f3}. The error equations can be obtained by subtracting the corresponding equations in \eqref{f1}-\eqref{f3} from \eqref{r1}-\eqref{r3} :
\begin{equation}\label{wucha1}
  \left\{
   \begin{aligned}
 &(\phi_{t}^{n+1}-\partial_{t}\phi_{h}^{n+1},u_{h})=-(\nabla (\mu^{n+1}-\mu_{h}^{n+1}),\nabla u_{h}), &&\forall u_{h}\in V_h,\\
 &(\mu^{n+1}-\mu_{h}^{n+1},v_{h})=(\nabla (\phi^{n+1}-\phi_{h}^{n+1}),\nabla v_{h})+\\
 &\qquad \qquad \qquad \qquad \left(r^{n+1}\frac{F'(\phi^{n+1})}{\sqrt{E(\phi^{n+1})}}-r_{h}^{n+1}\frac{F'(\phi_{h}^{n})}{\sqrt{E(\phi_{h}^{n})}},v_{h}\right),&&\forall v_{h}\in V_h,\\
 &r_{t}^{n+1}-\partial_{t}r_{h}^{n+1}=\left(\frac{F'(\phi^{n+1})}{2\sqrt{E(\phi^{n+1})}},\phi_{t}^{n+1}\right)-\left(\frac{F'(\phi_{h}^{n})}{2\sqrt{E(\phi_{h}^{n})}}, \partial_{t}\phi_{h}^{n+1}\right).
   \end{aligned}
   \right.
  \end{equation}

Considering the projection \eqref{t1}-\eqref{t3},the error equation \eqref{wucha1} can be rewritten as:
\begin{empheq}[left=\empheqlbrace]{alignat=2}
   & (\overline{\phi}_{t}^{n+1} - \partial_{t}\phi_{h}^{n+1}, u_h) = -(\nabla (\overline{\mu}^{n+1} - \mu_{h}^{n+1}), \nabla u_h), &&\qquad \forall u_h \in V_h, \label{wc1} \\
    &(\overline{\mu}^{n+1} - \mu_{h}^{n+1}, v_h) = (\nabla (\overline{\phi}^{n+1} - \phi_{h}^{n+1}), \nabla v_h) \nonumber \\
    &\qquad \qquad \qquad \qquad + \Bigl(r^{n+1}\frac{F'(\overline{\phi}^{n+1})}{\sqrt{E(\overline{\phi}^{n+1})}} - r_{h}^{n+1}\frac{F'(\phi_{h}^{n})}{\sqrt{E(\phi_{h}^{n})}}, v_h\Bigr), &&\qquad \forall v_h \in V_h, \label{wc2} \\
    &r_{t}^{n+1} - \partial_{t}r_{h}^{n+1} = \Bigl(\frac{F'(\phi^{n+1})}{2\sqrt{E(\phi^{n+1})}}, \phi_{t}^{n+1}\Bigr) - \Bigl(\frac{F'(\phi_{h}^{n})}{2\sqrt{E(\phi_{h}^{n})}}, \partial_{t}\phi_{h}^{n+1}\Bigr), \label{wc3}
\end{empheq}
\begin{theorem}\label{thm:error_estimate}
Assume that $\phi,\mu\in H^{2}(\Omega)\cap W^{1,\infty}(J;H^{1}(\Omega))\cap W^{2,\infty}(J;W^{1,\infty}(\Omega)),
r \in L^{\infty}(J)\cap W^{2,\infty}(J)$ and $F \in C^{2}(\mathbb{R})$ with $F''$ bounded on bounded sets.
Then for the fully discrete scheme \eqref{f1}-\eqref{f3}, there exists a positive constant $C$ independent of $h$ and $\Delta t$ such that for any $N \geq 0$ with $t^{N+1} = (N+1)\Delta t \leq T$,
\begin{align}
\|\phi^{N+1}-\phi_{h}^{N+1}\|_{1}\leq C(h+\Delta t),\label{phimur6}\\
\left(\sum_{n=0}^{N}\Delta t\|\mu^{n+\frac{1}{2}}-\mu_{h}^{n+\frac{1}{2}}\|_{1}^2\right)^{1/2}\leq C(h+\Delta t),\label{phimur7}\\
(r^{N+1}-r_{h}^{N+1})\leq C(h^{2}+\Delta t).\label{phimur8}
\end{align}
\end{theorem}

Proof. For simplicity, we set
\begin{align}
 &e_{\phi}^{n}=\phi^{n}-\phi_{h}^{n}=\phi^{n}-\overline{\phi}^{n}+\overline{\phi}^{n}-\phi_{h}^{n}=\rho_{\phi}^{n}+\theta_{\phi}^{n},\nonumber\\
 &e_{\mu}^{n}=\mu^{n}-\mu_{h}^{n}=\mu^{n}-\overline{\mu}^{n}+\overline{\mu}^{n}-\mu_{h}^{n}=\rho _{\mu}^{n}+\theta _{\mu}^{n},\nonumber\\
 &e_{r}^{n}=r^{n}-r_{h}^{n}.\nonumber
\end{align}

\eqref{wc1} can be rewritten as
\begin{align}
 (\partial_{t}\theta_{\phi}^{n+1},u_{h})=-(\overline{\phi}_{t}^{n+1}-\partial_{t}\overline{\phi}^{n+1},u_{h})-(\nabla \theta_{\mu}^{n+1},\nabla u_{h}),\label{phih11}
\end{align}
where
\begin{align}
\|\overline{\phi}_{t}^{n+1}-\partial_{t}\overline{\phi}^{n+1}\|\leq C\|\overline{\phi}\|_{W^{2,\infty}(J;L^{\infty}(\Omega))}\Delta t^{2}\leq C \|\phi\|_{W^{2,\infty}(J;L^{\infty}(\Omega))}\Delta t^{2}.
\end{align}

Let $u_{h}=\theta_{\phi}^{n+1}$ in \eqref{phih11},
\begin{align}
 \frac{\|\theta_{\phi}^{n+1}\|^2-\|\theta_{\phi}^{n}\|^2+\|\theta_{\phi}^{n+1}-\theta_{\phi}^{n}\|^2}{2\Delta t}=-(\overline{\phi}_{t}^{n+1}-\partial_{t}\overline{\phi}^{n+1},\theta_{\phi}^{n+1})-(\nabla \theta_{\mu}^{n+1},\nabla \theta_{\phi}^{n+1}).\label{phih12}
\end{align}

Let $u_{h}=\theta_{\mu}^{n+1}$ in \eqref{phih11},
\begin{align}
 (\partial_{t}\theta_{\phi}^{n+1},\theta_{\mu}^{n+1})=-(\overline{\phi}_{t}^{n+1}-\partial_{t}\overline{\phi}^{n+1},\theta_{\mu}^{n+1})-\|\nabla \theta_{\mu}^{n+1}\|^2.\label{phih13}
\end{align}

For all $x \in H^{1}(\Omega)$, let $x_{h} = \pi_{h} x$ in equation \eqref{phih11}. Applying Young's inequality, we obtain the following estimate:
\begin{align}
(\partial_{t}\theta_{\phi}^{n+1},x) &= (\partial_{t}\theta_{\phi}^{n+1},x_{h})= -(\overline{\phi}_{t}^{n+1}-\partial_{t}\overline{\phi}^{n+1},x_{h})-(\nabla \theta_{\mu}^{n+1},\nabla x_{h}) \nonumber \\
&\leq \|\overline{\phi}_{t}^{n+1}-\partial_{t}\overline{\phi}^{n+1}\|\|x_{h}\|+\|\nabla \theta_{\mu}^{n+1}\|\|\nabla x_{h}\| \nonumber \\
&\leq \left(\|\overline{\phi}_{t}^{n+1}-\partial_{t}\overline{\phi}^{n+1}\|+\|\nabla \theta_{\mu}^{n+1}\|\right)\|x_{h}\|_{1}\nonumber \\
&\leq \left(\|\overline{\phi}_{t}^{n+1}-\partial_{t}\overline{\phi}^{n+1}\|+\|\nabla \theta_{\mu}^{n+1}\|\right)\|x\|_{1}. \label{phih14}
\end{align}

Consequently, we derive the bound for the negative norm:
\begin{align}
\|\partial_{t}\theta_{\phi}^{n+1}\|_{-1} \leq& \|\overline{\phi}_{t}^{n+1}-\partial_{t}\overline{\phi}^{n+1}\|+\|\nabla \theta_{\mu}^{n+1}\|\nonumber\\
\leq& C \|\phi\|_{W^{2,\infty}(J;L^{\infty}(\Omega))}\Delta t^{2}+\|\nabla \theta_{\mu}^{n+1}\|.\label{phih15}
\end{align}

\eqref{wc2} can be rewritten as
\begin{align}
 (\theta_{\mu}^{n+1},v_{h})=&(\nabla \theta_{\phi}^{n+1},\nabla v_{h})+
 \left(e_{r}^{n+1} \frac{F'(\phi_{h}^{n})}{\sqrt{E(\phi_{h}^{n})}}+r^{n+1}\left(\frac{F'(\overline{\phi}^{n+1})}{\sqrt{E(\overline{\phi}^{n+1})}}-\frac{F'(\phi_{h}^{n})}{\sqrt{E(\phi_{h}^{n})}}\right),v_{h}\right).\label{muh21}
\end{align}

Let $v_{h}=\partial_{t}\theta_{\phi}^{n+1}$ in \eqref{muh21},
\begin{align}
 (\theta_{\mu}^{n+1},\partial_{t}\theta_{\phi}^{n+1})=&\frac{\|\nabla\theta_{\phi}^{n+1}\|^{2}-\|\nabla\theta_{\phi}^{n}\|^{2}+\|\nabla(\theta_{\phi}^{n+1}-\theta_{\phi}^{n})\|^{2}}{2 \Delta t}+\left(e_{r}^{n+1}\frac{F'(\phi_{h}^{n})}{\sqrt{E(\phi_{h}^{n})}},\partial_{t}\theta_{\phi}^{n+1}\right)\nonumber\\
 &+\left(r^{n+1}\left(\frac{F'(\overline{\phi}^{n+1})}{\sqrt{E(\overline{\phi}^{n+1})}}-\frac{F'(\phi_{h}^{n})}{\sqrt{E(\phi_{h}^{n})}}\right),\partial_{t}\theta_{\phi}^{n+1}\right).\label{muh22}
\end{align}

Let $v_{h}=\theta_{\mu}^{n+1}$ in \eqref{muh21},
\begin{align}
 \|\theta_{\mu}^{n+1}\|^{2}=&(\nabla \theta_{\phi}^{n+1},\nabla \theta_{\mu}^{n+1})+
 \left(e_{r}^{n+1}\frac{F'(\phi_{h}^{n})}{\sqrt{E(\phi_{h}^{n})}},\theta_{\mu}^{n+1}\right)\nonumber\\
 &+\left(r^{n+1}\left(\frac{F'(\overline{\phi}^{n+1})}{\sqrt{E(\overline{\phi}^{n+1})}}-\frac{F'(\phi_{h}^{n})}{\sqrt{E(\phi_{h}^{n})}}\right),\theta_{\mu}^{n+1}\right).\label{muh23}
\end{align}

\eqref{wc3} can be rewritten as
\begin{align}
 \partial_{t}e_{r}^{n+1}=&-(r_{t}^{n+1}-\partial_{t}r^{n+1})+\left(\frac{F'(\phi^{n+1})}{2\sqrt{E(\phi^{n+1})}},\phi_{t}^{n+1}-\partial_{t}\phi^{n+1}\right)+\nonumber\\
&\left(\frac{F'(\phi^{n+1})}{2\sqrt{E(\phi^{n+1})}}, \partial_{t}\theta_{\phi}^{n+1}\right)+\left(\frac{F'(\phi^{n+1})}{2\sqrt{E(\phi^{n+1})}}-\frac{F'(\phi_{h}^{n})}{2\sqrt{E(\phi_{h}^{n})}},\partial_{t}\phi^{n+1}_{h}\right),\label{rh31}
\end{align}
where
\begin{align}
|r_{t}-\partial_{t}r^{n+1}|\leq  C \|r\|_{W^{2,\infty}(J)}\Delta t^{2}.
\end{align}

Multiplying equation \eqref{rh31} by $2 e_{r}^{n+1}$, we have
\begin{align}
&\frac{(e_{r}^{n+1})^2-(e_{r}^{n})^2+(e_{r}^{n+1}-e_{r}^{n})^2}{\Delta t}\nonumber\\
=&-2(r_{t}^{n+1}-\partial_{t}r^{n+1})e_{r}^{n+1}+\left(e_{r}^{n+1} \frac{F'(\phi^{n+1})}{\sqrt{E(\phi^{n+1})}},\phi_{t}^{n+1}-\partial_{t}\phi^{n+1}\right)+\nonumber\\
&\left(e_{r}^{n+1}\frac{F'(\phi^{n+1})}{\sqrt{E(\phi^{n+1})}}, \partial_{t}\theta_{\phi}^{n+1}\right)+\left(e_{r}^{n+1} \left(\frac{F'(\phi^{n+1})}{\sqrt{E(\phi^{n+1})}}-\frac{F'(\phi_{h}^{n})}{\sqrt{E(\phi_{h}^{n})}}\right),\partial_{t}\phi^{n+1}_{h}\right),\label{rh32}
\end{align}
where
\begin{align}
\|\phi_{t}-\partial_{t}\phi^{n+1}\|\leq C \|\phi\|_{W^{2,\infty}(J;L^{\infty}(\Omega))}\Delta t^{2}.
\end{align}

Together with $\eqref{phih12},\eqref{phih13},\eqref{phih14},\eqref{muh22},\eqref{muh23}$ and $\eqref{rh32}$, we can derive
\begin{align}
&\frac{\|\theta_{\phi}^{n+1}\|_{1}^{2}-\|\theta_{\phi}^{n}\|_{1}^{2}+\|\theta_{\phi}^{n+1}-\theta_{\phi}^{n}\|_{1}^{2}}{2 \Delta t}+\|\theta_{\mu}^{n+1}\|_{1}^{2}+\frac{(e_{r}^{n+1})^2-(e_{r}^{n})^2+(e_{r}^{n+1}-e_{r}^{n})^2}{\Delta t}\nonumber\\
=&-(\overline{\phi}_{t}^{n+1}-\partial_{t}\overline{\phi}^{n+1},\theta_{\phi}^{n+1})-(\overline{\phi}_{t}^{n+1}-\partial_{t}\overline{\phi}^{n+1},\theta_{\mu}^{n+1})-\nonumber\\
&\left(r^{n+1}\left(\frac{F'(\overline{\phi}^{n+1})}{\sqrt{E(\overline{\phi}^{n+1})}}-\frac{F'(\phi_{h}^{n})}{\sqrt{E(\phi_{h}^{n})}}\right),\partial_{t}\theta_{\phi}^{n+1}\right)+
 \left(e_{r}^{n+1}\frac{F'(\phi_{h}^{n})}{\sqrt{E(\phi_{h}^{n})}},\theta_{\mu}^{n+1}\right)\nonumber\\
 &+\left(r^{n+1}\left(\frac{F'(\overline{\phi}^{n+1})}{\sqrt{E(\overline{\phi}^{n+1})}}-\frac{F'(\phi_{h}^{n})}{\sqrt{E(\phi_{h}^{n})}}\right),\theta_{\mu}^{n+1}\right)-2 (r_{t}^{n+1}-\partial_{t}r^{n+1}) e_{r}^{n+1}\nonumber\\
 &+\left(e_{r}^{n+1}\frac{F'(\phi^{n+1})}{\sqrt{E(\phi^{n+1})}},\phi_{t}^{n+1}-\partial_{t}\phi^{n+1}\right)+\nonumber\\
&\left(e_{r}^{n+1}\left(\frac{F'(\phi^{n+1})}{\sqrt{E(\phi^{n+1})}}-\frac{F'(\phi_{h}^{n})}{\sqrt{E(\phi_{h}^{n})}}\right),\partial_{t}\phi^{n+1}_{h}\right):=\sum_{i=1}^{8}A_i.\label{phimur1}
\end{align}
\begin{align}
A_{1}=&-(\overline{\phi}_{t}^{n+1}-\partial_{t}\overline{\phi}^{n+1},\theta_{\phi}^{n+1})\leq\|\overline{\phi}_{t}^{n+1}-\partial_{t}\overline{\phi}^{n+1}\|^{2}+\|\theta_{\phi}^{n+1}\|^{2}\nonumber\\
\leq&C \|\overline{\phi}\|_{W^{2,\infty}(J;L^{\infty}(\Omega))}^{2}\Delta t^{4}+\|\theta_{\phi}^{n+1}\|^{2}\leq C \|\phi\|_{W^{2,\infty}(J;L^{\infty}(\Omega))}^{2}\Delta t^{4}+\|\theta_{\phi}^{n+1}\|_{1}^{2}.\label{A1}
\end{align}
\begin{align}
A_{2}=&-(\overline{\phi}_{t}^{n+1}-\partial_{t}\overline{\phi}^{n+1},\theta_{\mu}^{n+1})\leq 2 \|\overline{\phi}_{t}^{n+1}-\partial_{t}\overline{\phi}^{n+1}\|^{2}+\frac{1}{8} \|\theta_{\mu}^{n+1}\|^{2}\nonumber\\
\leq& C \|\overline{\phi}\|_{W^{2,\infty}(J;L^{\infty}(\Omega))}^{2}\Delta t^{4}+\frac{1}{8}\|\theta_{\mu}^{n+1}\|^{2}\leq C \|\phi\|_{W^{2,\infty}(J;L^{\infty}(\Omega))}^{2}\Delta t^{4}+\frac{1}{8}\|\theta_{\mu}^{n+1}\|_{1}^{2}.\label{A2}
\end{align}

Considering
\begin{align}
&\left\|r^{n+1}\left(\frac{F'(\overline{\phi}^{n+1})}{\sqrt{E(\overline{\phi}^{n+1})}}-\frac{F'(\phi_{h}^{n})}{\sqrt{E(\phi_{h}^{n})}}\right)\right\|\nonumber\\
\leq&\|r\|_{L^{\infty}(J)}\left(\left\|\frac{F'(\overline{\phi}^{n+1})}{\sqrt{E(\overline{\phi}^{n+1})}}-\frac{F'(\phi_{h}^{n+1})}{\sqrt{E(\phi_{h}^{n+1})}}\right\|+\left\|\frac{F'(\phi_{h}^{n+1})}{\sqrt{E(\phi_{h}^{n+1})}}-\frac{F'(\phi_{h}^{n})}{\sqrt{E(\phi_{h}^{n+1})}}\right\|\right)\nonumber\\
\leq & C \|r\|_{L^{\infty}(J)}\left(\|\theta_{\phi}^{n+1}\|+\|\phi_{h}^{n+1}-\phi_{h}^{n}\|\right)\nonumber\\
\leq & C \|r\|_{L^{\infty}(J)}\left(\|\theta_{\phi}^{n+1}\|_{1}+\|\phi\|_{W^{1,\infty}(J,H^{1}(\Omega))}\Delta t\right),\label{FF1}
\end{align}
similarly,\begin{align}
\left\|r^{n+1}\left(\frac{F'(\overline{\phi}^{n+1})}{\sqrt{E(\overline{\phi}^{n+1})}}-\frac{F'(\phi_{h}^{n})}{\sqrt{E(\phi_{h}^{n})}}\right)\right\|_{1}\leq C\|r\|_{L^{\infty}(J)}\left(\|\theta_{\phi}^{n+1}\|_{1}+\|\phi\|_{W^{1,\infty}(J,H^{1}(\Omega))}\Delta t\right).\label{FF2}
\end{align}

Considering (\ref{phih15}) and (\ref{FF2}),
\begin{align}
A_{3} =& -\left(r^{n+1}\left(\frac{F'(\overline{\phi}^{n+1})}{\sqrt{E(\overline{\phi}^{n+1})}}-\frac{F'(\phi_{h}^{n})}{\sqrt{E(\phi_{h}^{n})}}\right),\partial_{t}\theta_{\phi}^{n+1}\right)\nonumber\\
\leq& 2 \left\|r^{n+1}\left(\frac{F'(\overline{\phi}^{n+1})}{\sqrt{E(\overline{\phi}^{n+1})}}-\frac{F'(\phi_{h}^{n})}{\sqrt{E(\phi_{h}^{n})}}\right)\right\|_{1}^{2}+\frac{1}{8}\|\partial_{t}\theta_{\phi}^{n+1}\|_{-1}^{2}\nonumber\\
\leq& C\|r\|_{L^{\infty}(J)}^{2}\left(\|\theta_{\phi}^{n+1}\|_{1}^{2}+\|\phi\|_{W^{1,\infty}(J,H^{1}(\Omega))}^2\Delta t^{2}\right)+\nonumber\\
&\frac{1}{8}\left(C \|\phi\|_{W^{2,\infty}(J;L^{\infty}(\Omega))}^2\Delta t^{4}+\|\nabla \theta_{\mu}^{n+1}\|^2\right)\nonumber\\
\leq& C \|r\|_{L^{\infty}(J)}^{2}\|\theta_{\phi}^{n+1}\|_{1}^{2}+C\left( \|r\|_{L^{\infty}(J)}^{2}\|\phi\|_{W^{1,\infty}(J,H^{1}(\Omega))}^2\right.\nonumber\\
&\left.+\|\phi\|_{W^{2,\infty}(J;L^{\infty}(\Omega))}^2\Delta t^{2}\right)\Delta t^{2}+\frac{1}{8}\| \theta_{\mu}^{n+1}\|_1^2.\label{FF3}
\end{align}
\begin{align}
A_4=&\left(e_{r}^{n+1}\frac{F'(\phi_{h}^{n})}{\sqrt{E(\phi_{h}^{n})}},\theta_{\mu}^{n+1}\right)\leq  2 \left\|e_{r}^{n+1}\frac{F'(\phi_{h}^{n})}{\sqrt{E(\phi_{h}^{n})}}\right\|^{2}+\frac{1}{8}\left\|\theta_{\mu}^{n+1}\right\|^{2}\nonumber\\
\leq& C (e_{r}^{n+1})^{2}+\frac{1}{8}\left\|\theta_{\mu}^{n+1}\right\|_{1}^{2}.\label{A4}
\end{align}
Considering (\ref{FF1}),
\begin{align}
A_5=&\left(r^{n+1}\left(\frac{F'(\overline{\phi}^{n+1})}{\sqrt{E(\overline{\phi}^{n+1})}}-\frac{F'(\phi_{h}^{n})}{\sqrt{E(\phi_{h}^{n})}}\right),\theta_{\mu}^{n+1}\right)\nonumber\\
\leq & 2 \left\|r^{n+1}\left(\frac{F'(\overline{\phi}^{n+1})}{\sqrt{E(\overline{\phi}^{n+1})}}-\frac{F'(\phi_{h}^{n})}{\sqrt{E(\phi_{h}^{n})}}\right)\right\|^{2}+\frac{1}{8}\left\|\theta_{\mu}^{n+1}\right\|^{2}\nonumber\\
\leq & C \|r\|_{L^{\infty}(J)}^{2}\left(\|\theta_{\phi}^{n+1}\|_{1}^{2}+\|\phi\|_{W^{1,\infty}(J,H^{1}(\Omega))}^2\Delta t^{2}\right)+\frac{1}{8}\left\|\theta_{\mu}^{n+1}\right\|_{1}^{2}.
\end{align}
\begin{align}
A_6=&-2 (r_{t}^{n+1}-\partial_{t}r^{n+1})\cdot e_{r}^{n+1}\leq C (r_{t}^{n+1}-\partial_{t}r^{n+1})^{2}+C(e_{r}^{n+1})^{2}\nonumber\\
\leq & C \|r\|_{W^{2,\infty}(J)}^{2} \Delta t^{4}+C (e_{r}^{n+1})^{2}.\label{A6}
\end{align}
\begin{align}
A_7=&\left(e_{r}^{n+1}\frac{F'(\phi^{n+1})}{\sqrt{E(\phi^{n+1})}},\phi_{t}^{n+1}-\partial_{t}\phi^{n+1}\right)\leq \left\|e_{r}^{n+1}\frac{F'(\phi^{n+1})}{\sqrt{E(\phi^{n+1})}}\right\|^{2}+ \left\|\phi_{t}^{n+1}-\partial_{t}\phi^{n+1}\right\|^{2}\nonumber\\
\leq & C (e_{r}^{n+1})^{2}+C \left\|\phi\right\|_{W^{2,\infty}(J,L^{\infty}(\Omega))}^{2}\Delta t^{4}.\label{A7}
\end{align}
\begin{align}
A_8=&\left(e_{r}^{n+1}\left(\frac{F'(\phi^{n+1})}{\sqrt{E(\phi^{n+1})}}-\frac{F'(\phi_{h}^{n})}{\sqrt{E(\phi_{h}^{n})}}\right),\partial_{t}\phi^{n+1}_{h}\right)\nonumber\\
\leq &  \left|e_{r}^{n+1}\right|\left\|\frac{F'(\phi^{n+1})}{\sqrt{E(\phi^{n+1})}}-\frac{F'(\phi_{h}^{n})}{\sqrt{E(\phi_{h}^{n})}}\right\|\left\|\partial_{t}\phi^{n+1}_{h}\right\|\nonumber\\
\leq & C \|\phi\|_{W^{1,\infty}(J,H^{1}(\Omega))}\left((e_{r}^{n+1})^{2}+\|\phi^{n+1}-\phi_{h}^{n+1}\|^{2}+\|\phi_{h}^{n+1}-\phi_{h}^{n}\|^{2}\right)\nonumber\\
\leq & C \|\phi\|_{W^{1,\infty}(J,H^{1}(\Omega))}\left((e_{r}^{n+1})^{2}+\|\rho_{\phi}^{n+1}\|^{2}+\|\theta_{\phi}^{n+1}\|_{1}^{2}+\|\phi\|_{W^{1,\infty}(J,H^{1}(\Omega))}^2\Delta t^{2}\right).\label{A8}
\end{align}

Based on the above analysis, equation\eqref{phimur1} can be estimated by
\begin{align}
&\frac{\|\theta_{\phi}^{n+1}\|_{1}^{2}-\|\theta_{\phi}^{n}\|_{1}^{2}}{2 \Delta t}+\|\theta_{\mu}^{n+1}\|_{1}^{2}+\frac{(e_{r}^{n+1})^2-(e_{r}^{n})^2}{\Delta t}\nonumber\\
\leq&C\|\phi\|_{W^{1,\infty}(J;H^{1}(\Omega))}^{2} \|\rho_{\phi}^{n+1}\|^{2}+C \left(\|\phi\|_{W^{2,\infty}(J;L^{\infty}(\Omega))}^{2}\Delta t^{2}+\|r\|_{L^{\infty}(J)}^{2}\|\phi\|_{W^{1,\infty}(J;H^{1}(\Omega))}^{2}\right.\nonumber\\
&\left.+\|\phi\|_{W^{1,\infty}(J;H^{1}(\Omega))}^{2}+\|r\|_{W^{2,\infty}(J)}^{2}\Delta t^{2}\right)\Delta t^{2}+C \left(\|\phi\|_{W^{1,\infty}(J;H^{1}(\Omega))}^{2}+\right.\nonumber\\
&\left.\|r\|_{L^{\infty}(J)}^{2}\right)\|\theta_{\phi}^{n+1}\|_{1}^2+C \|\phi\|_{W^{1,\infty}(J;L^{\infty}(\Omega))}^{2}(e_{r}^{n+1})^2+\frac{1}{2}\|\theta_{\mu}^{n+1}\|_{1}^2\nonumber\\
\leq &C (h^4+\Delta t^{2})+C \left(\|\theta_{\phi}^{n+1}\|_{1}^2+(e_{r}^{n+1})^2\right)+\frac{1}{2}\|\theta_{\mu}^{n+1}\|_{1}^2.\label{phimur3}
\end{align}
Assuming $\theta_{\phi}^{0}=0,\theta_{\mu}^{0}=0,e_{r}^{0}=0$, multiplying equation \eqref{phimur3} by $2 \Delta t$, summing over $n,n=0,1,2,...,N$, we can obtain
\begin{align}
&\|\theta_{\phi}^{N+1}\|_{1}^{2}+\sum_{n=0}^{N}\Delta t\|\theta_{\mu}^{n+1}\|_{1}^{2}+2 (e_{r}^{N+1})^2\nonumber\\
\leq&C(h^4+\Delta t^{2})+C\sum_{n=0}^{N}\Delta t\left(\|\theta_{\phi}^{n+1}\|_{1}^2+(e_{r}^{n+1})^2\right).\label{phimur4}
\end{align}
It follows from Gronwall's inequality that
\begin{align}
\|\theta_{\phi}^{N+1}\|_{1}^{2}+\sum_{n=0}^{N}\Delta t\|\theta_{\mu}^{n+1}\|_{1}^{2}+(e_{r}^{N+1})^2\leq C(h^{4}+\Delta t^{2}).\label{phimur5}
\end{align}
Thus, we can obtain the desired result \eqref{phimur6}-\eqref{phimur8}. $\qquad\Box$
\subsection{Algorithm Design}
The fully discrete scheme (\ref{f1})-(\ref{f3}) constitutes a coupled linear system for the variables $(\phi_{h}^{n}, \mu_{h}^{n},  r_{h}^n)$. While standard iterative or direct solvers can be employed to solve this coupled linear system, this approach is often not the most efficient. To address this, we reformulate the variables $(\phi_{h}^{n}, \mu_{h}^{n})$ as
\begin{align}\label{fenl}
(\phi_{h}^{n}, \mu_{h}^{n})=(\phi_{1h}^{n}, \mu_{1h}^{n})+r_{h}^{n}(\phi_{2h}^{n}, \mu_{2h}^{n}),
\end{align}
where $(\phi_{ih}^{n}, \mu_{ih}^{n}) (i=1,2)$ can be obtained by solving the following decoupled system of equations, respectively.
\begin{equation}\label{jieou1}
  \left\{
   \begin{aligned}
&(\phi_{1h}^{n+1},u_{h})+\Delta t (\nabla \mu_{1h}^{n+1},\nabla u_{h})=(\phi_{1h}^{n},u_{h}),\\
&(\mu_{1h}^{n+1},v_{h})-(\nabla \phi_{1h}^{n+1},\nabla v_{h})=0,
   \end{aligned}
   \right.
  \end{equation}
\begin{equation}\label{jieou2}
  \left\{
   \begin{aligned}
&(\phi_{2h}^{n+1},u_{h})+\Delta t(\nabla \mu_{2h}^{n+1},\nabla u_{h})=(\phi_{2h}^{n},u_{h}),\\
&(\mu_{2h}^{n+1},v_{h})-(\nabla \phi_{2h}^{n+1},\nabla v_{h})=\left(\frac{F'(\phi_{h}^{n})}{\sqrt{E(\phi_{h}^{n})}},v_{h}\right),
   \end{aligned}
   \right.
  \end{equation}
The system of equations (\ref{jieou1})-(\ref{jieou2}) is not only linear, but also features constant coefficient matrices, enabling the efficient solution of $(\phi_{ih}^{n+1}, \mu_{ih}^{n+1}) (i=1,2)$.

By comprehensively considering equations (\ref{f3}) and (\ref{fenl}), we obtain the following linear equation for $r_{h}^{n+1}$:
\begin{equation}
  \begin{aligned}
r_{h}^{n+1}-\left(\frac{F'(\phi_{h}^{n})}{2 \sqrt{E(\phi_{h}^{n})}}, \phi_{1h}^{n+1} + r_{h}^{n+1}\phi_{2h}^{n+1} \right) = r_{h}^{n} - \left( \frac{F'(\phi_{h}^{n})}{2\sqrt{E(\phi_{h}^{n})}}, \phi_{h}^{n} \right)
\end{aligned}
  \end{equation}
From this, we can compute
\begin{equation}
  \begin{aligned}
r_{h}^{n+1}=\frac{\left(\frac{F'(\phi_{h}^{n})}{2 \sqrt{E(\phi_{h}^{n})}}, \phi_{1h}^{n+1} -\phi_{h}^{n} \right)+r_{h}^{n}}{1-\left(\frac{F'(\phi_{h}^{n})}{2 \sqrt{E(\phi_{h}^{n})}}, \phi_{2h}^{n+1}\right)}.
\end{aligned}
  \end{equation}
Substituting the computed results into equation (\ref{fenl}) then allows us to determine $(\phi_{h}^{n+1}, \mu_{h}^{n+1})$.

\section{Numerical Examples and Results}
\label{sec:numerical}
This section presents a numerical investigation to validate the key properties of the proposed SAV scheme. Through a carefully designed test case, we focus on two critical aspects: first, verifying the unconditional decay of the modified free energy, which underpins the scheme¡¯s stability; second, quantifying the experimental errors to confirm the theoretical convergence rates established in Section \ref{sec:fully discrete}. For simplicity, we define
\begin{align}
 &\|u-v\|_{\infty,1}=\max_{0\leq n\leq N}\{\|u^{n}-v^{n} \|_{1}\},\nonumber\\
 &\|u-v\|_{2,1}=\left(\sum_{n=0}^{N} \Delta t \|u^{n}-v^{n}\|_{1}^{2}\right)^{1/2},\nonumber\\
 &\|e_{r}\|_{\infty}=\max_{0\leq n\leq N}\{|r^{n}-r_{h}^{n} |\}.\nonumber
\end{align}

The computational domain is taken as $\Omega=[0,1]\times[0,1]$, which is discretized using a uniform triangular mesh. The classical double-well potential $$F(\phi)=\frac{1}{4\epsilon^2}(\phi^{2}-1)^{2}, \ \epsilon=0.1$$ is employed, with the phase-field initial condition set to $$\phi(x,y,0)=0.05\sin(2\pi x)\sin(2\pi y).$$
\begin{figure*}[!htbp]
\centering
\subfigure[Variation of the mass with time] {\includegraphics[height=2.0in,width=2.5in,angle=0]{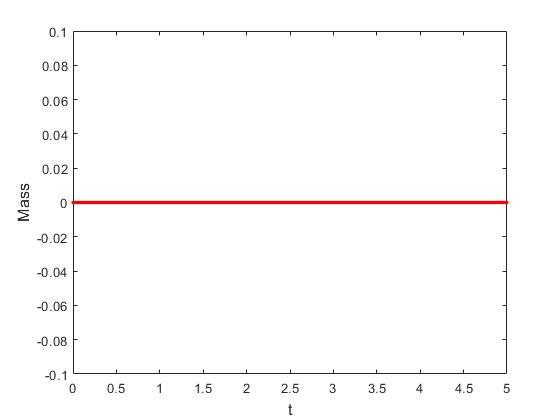}}
\subfigure[Variation of the total free energy with time] {\includegraphics[height=2.0in,width=2.5in,angle=0]{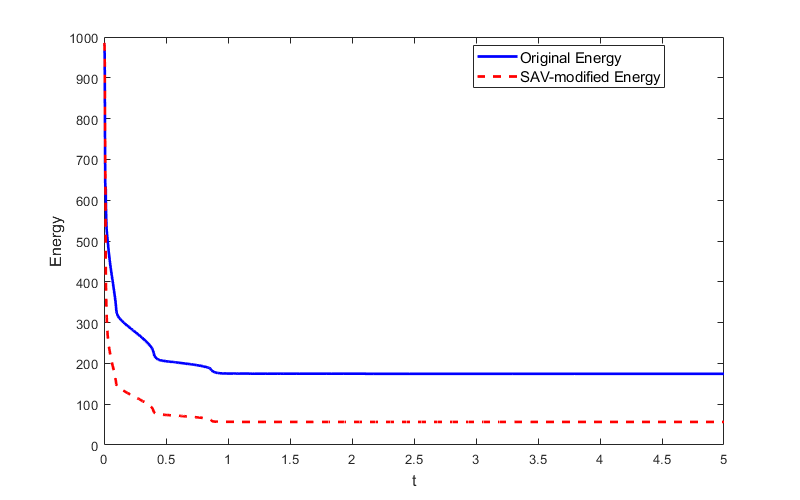}}
\caption{Temporal Evolution of Mass and Energy}\label{EM}
\end{figure*}

Figure \ref{EM} depicts the time evolution of the mass and total free energy over the interval $t=\in [0,5]$, computed with mesh size $h=1/16$ and time step $\Delta t=1/1000$. Subfigure (a) demonstrates that the mass remains virtually constant throughout the simulation, thereby verifying the mass-conservation property of the SAV scheme. Subfigure (b) presents the temporal evolution of both the original free energy and the SAV-modified free energy: both quantities decrease monotonically, exhibiting an initially rapid decay that gradually slows down. Notably, the SAV-modified free energy decays slightly faster than the original energy, consistent with the unconditional energy stability of the numerical scheme (\ref{energysh}) and (\ref{rff1}); Both energies approach steady-state values by approximately $t\approx 1.0$.
\begin{figure*}[!htbp]
\centering
\subfigure[$t=0$] {\includegraphics[height=1.35in,width=1.85in,angle=0]{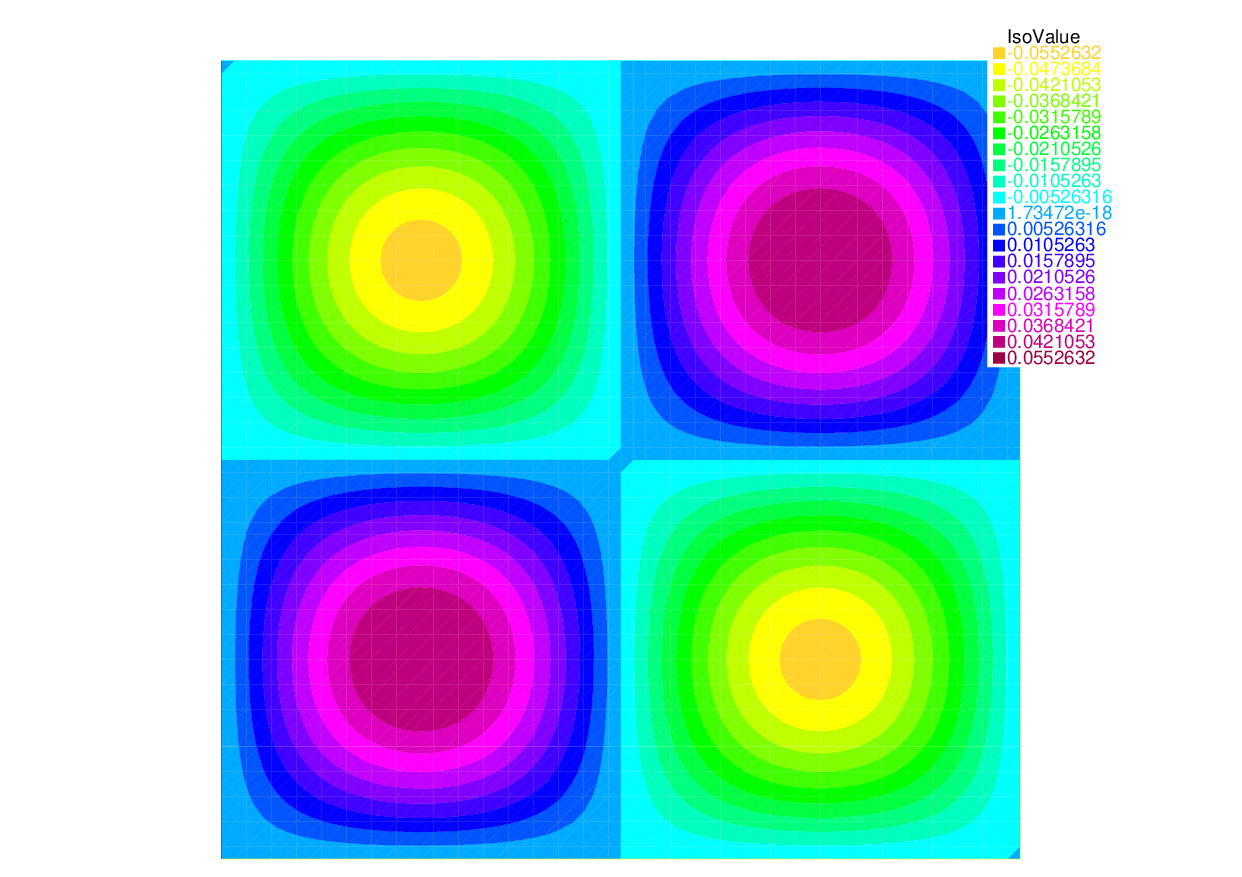}}
\subfigure[$t=0.05$] {\includegraphics[height=1.35in,width=1.85in,angle=0]{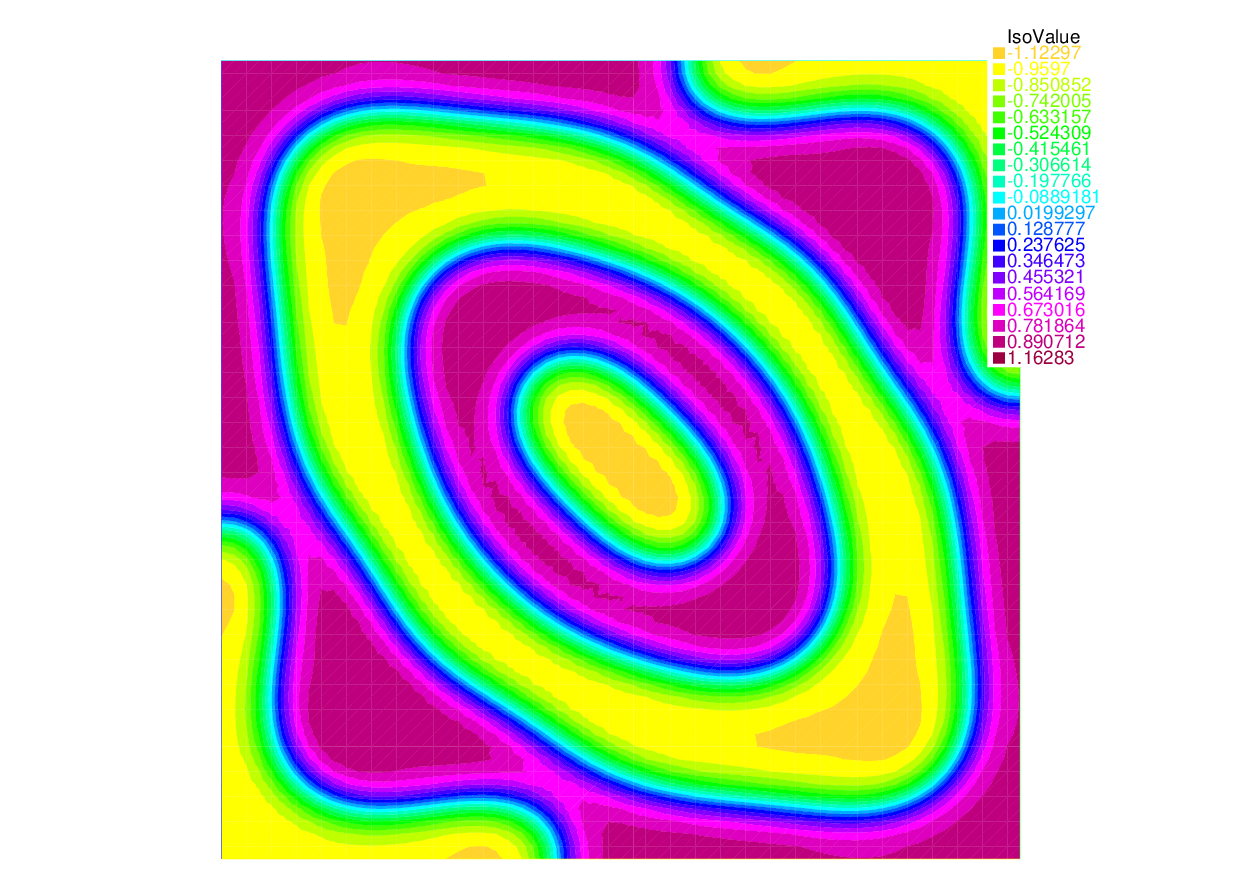}}
\subfigure[$t=0.1$] {\includegraphics[height=1.35in,width=1.85in,angle=0]{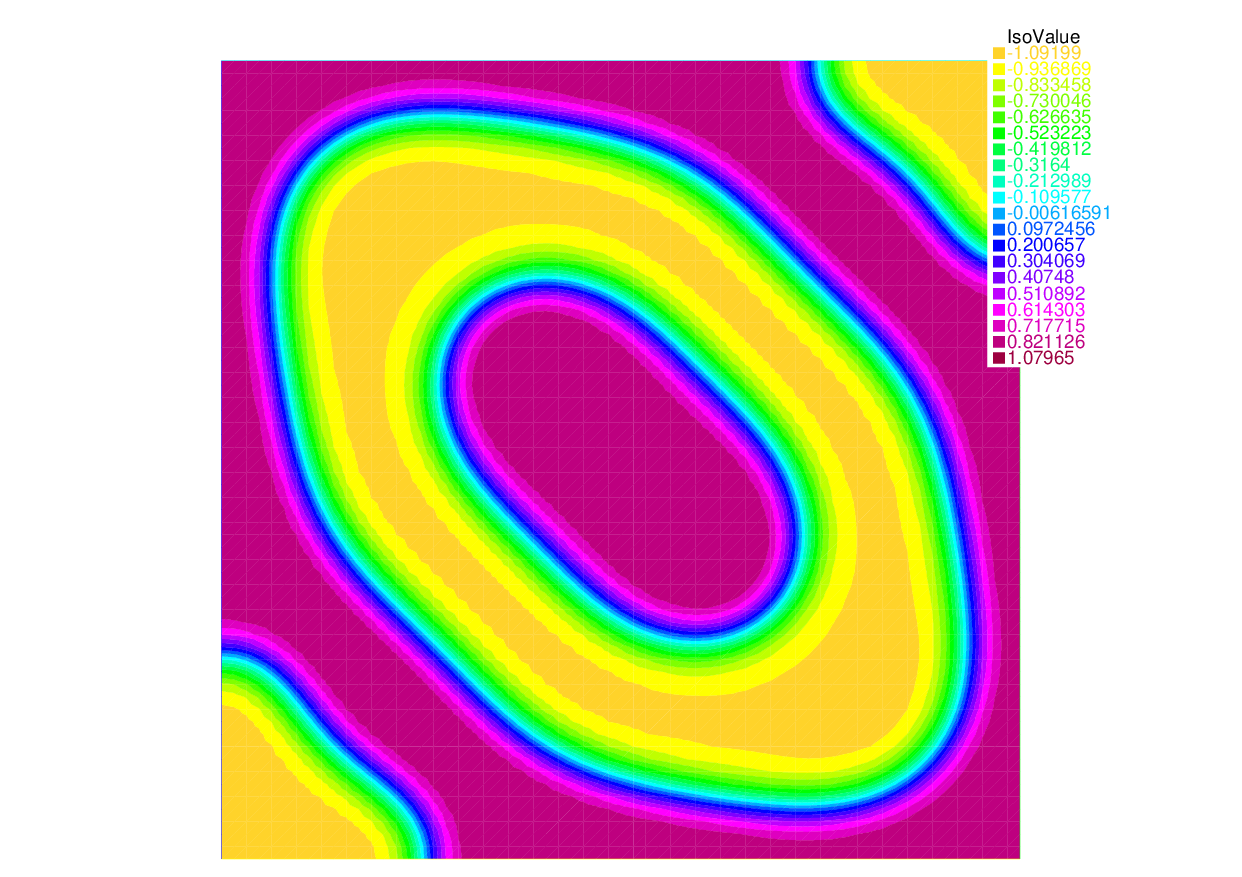}}
\subfigure[$t=0.3$] {\includegraphics[height=1.35in,width=1.85in,angle=0]{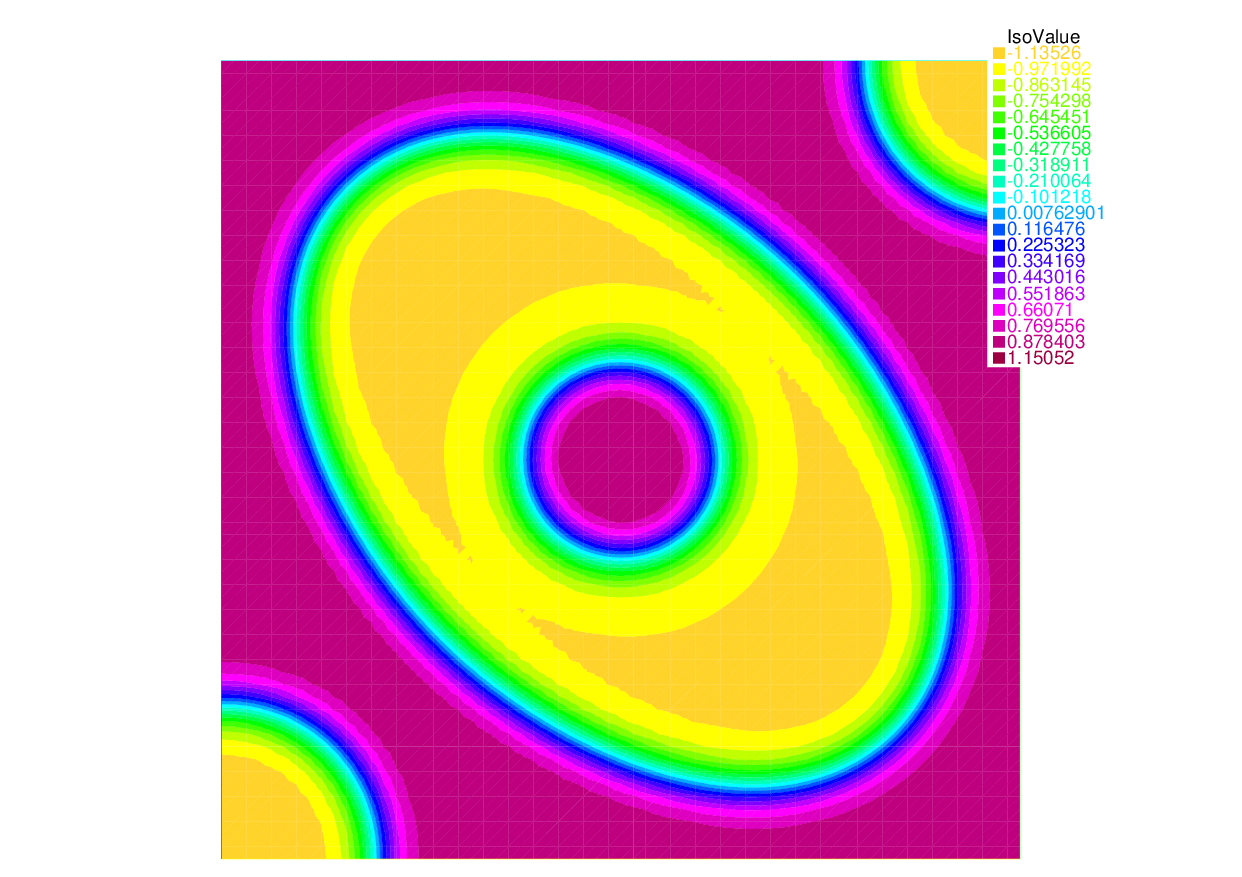}}
\subfigure[$t=0.8$] {\includegraphics[height=1.35in,width=1.85in,angle=0]{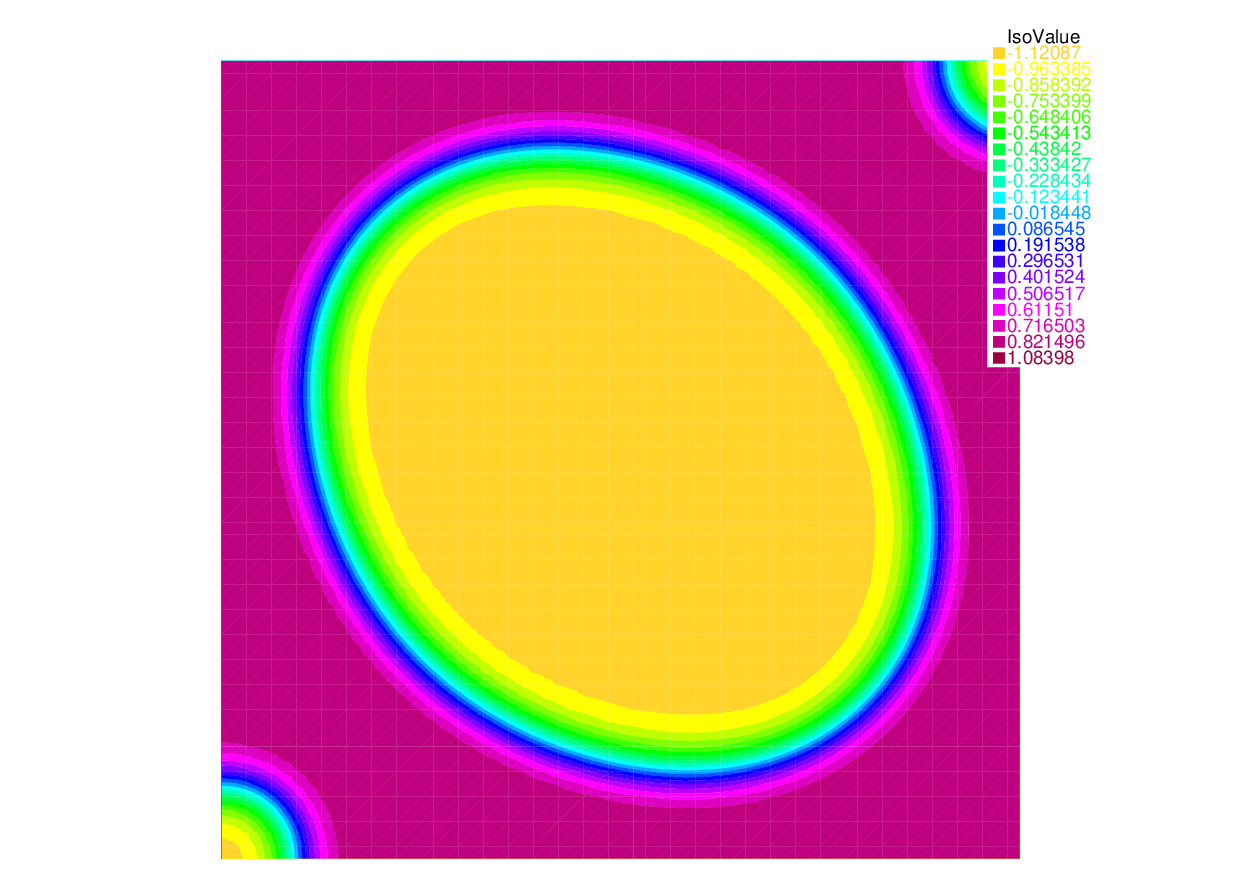}}
\subfigure[$t=2$] {\includegraphics[height=1.35in,width=1.85in,angle=0]{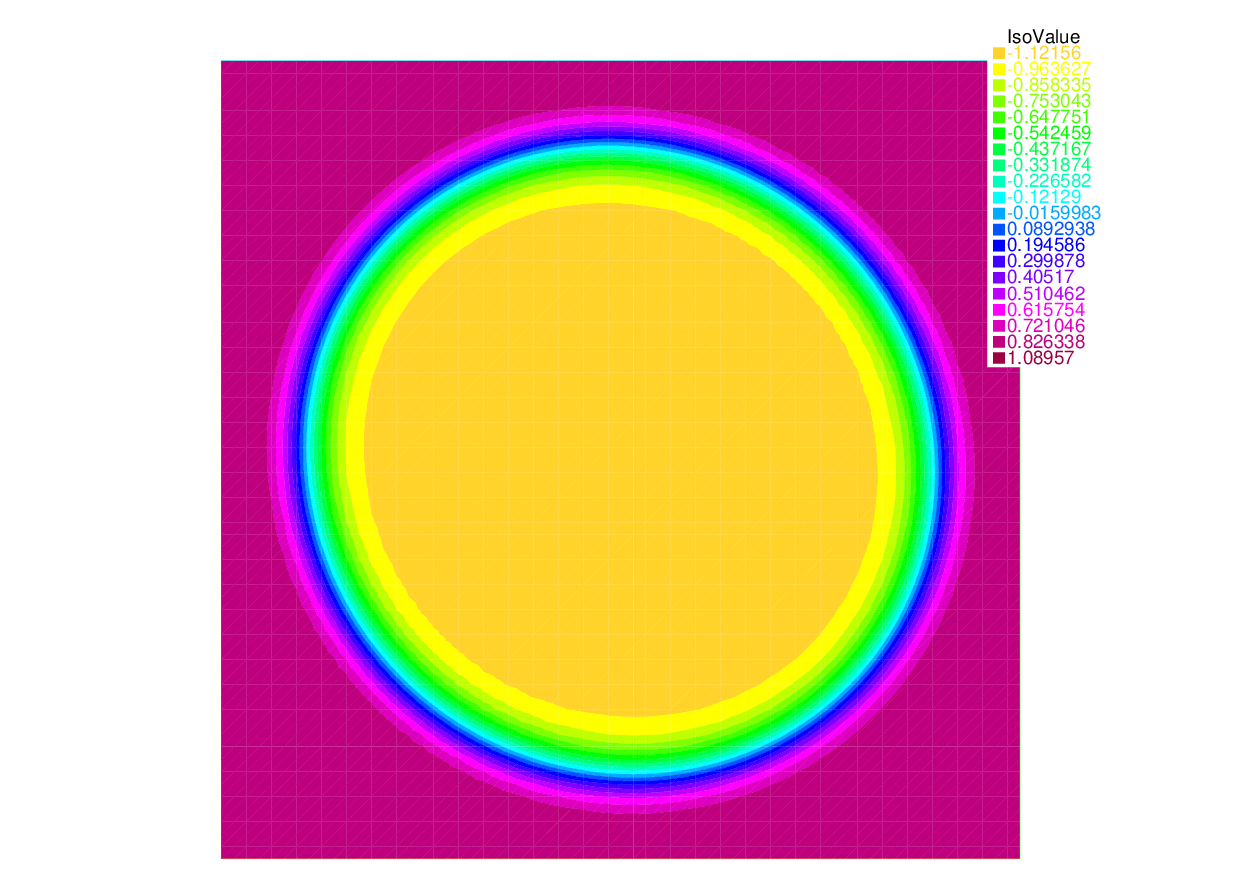}}
\caption{Evolution of the phase field variable with time}
\label{phi}
\end{figure*}

Figure \ref{phi} illustrates the evolution of the phase-field variable $\phi$ over time $t$. At the initial moment ($t=0$), the system is in a homogeneous or nearly homogeneous state, possibly with small random perturbations. The phase-field variable is uniformly distributed in space, indicating that the system is in a metastable or uniformly mixed state. In the early stage ($t=0.05,\ 0.1$), due to thermodynamic instability, the uniform state becomes unstable, and small-amplitude spatial modulations begin to emerge. The phase-field variable gradually differentiates, forming local enriched regions ($\phi>0$) and depleted regions ($\phi<0$), corresponding to the onset of phase separation. At this point, the structural scale is small, and the interfaces are relatively diffuse. During the intermediate stage ($t=0.3,\ 0.8$), phase separation becomes more pronounced. The enriched regions gradually merge and undergo coarsening, forming clear phase-domain structures. The interfaces become sharp, consistent with the tendency of interface energy minimization in the Cahn-Hilliard equation. In the late stage ($t=2$), the system gradually approaches a steady state or dynamic equilibrium. The phase domains further coarsen, the structural scale increases, and the interfaces become smoother, exhibiting a quasi-static pattern.

Since no analytical solution is available for this example, the numerical solutions $\phi_h$ and $\mu_h$ obtained with $h=1/128$ and $\Delta t=1/320000$ are taken as the reference solutions $\phi$ and $\mu$, respectively. Table \ref{biao1} lists the temporal error norms and convergence orders for $\phi_h$, $\mu_h$, and $r_h$ with respect to the reference solutions at $T=0.01$ under $h=1/128$. The results show that the errors of $\phi_h$, $\mu_h$, and $r_h$ are quite small and all achieve first-order accuracy in time. Table \ref{biao2} presents the spatial error norms and convergence orders for $\phi_h$, $\mu_h$, and $r_h$ at $T=0.01$ with $\Delta t=1/320000$. It is found that $\phi_h$ and $\mu_h$ exhibit first-order convergence, while $r_h$ shows second-order convergence. These observations agree well with the theoretical predictions stated in Theorem~\ref{thm:error_estimate}.
\begin{table}[!htb]
\centering
\caption{Time convergence rates for $\phi$, $\mu$ and $r$ with $h=\frac{1}{128}$ at $T=0.01$.}
\begin{tabular}{c| c c| c c| c c}
\hline
$\Delta t$ & $\|e_{\phi}\|_{\infty,1}$ & Rate  & $\|e_{\mu}\|_{\infty,1}$ & Rate & $\|e_{r}\|_{\infty}$ & Rate \\
\hline
$1/10000$ & 2.92694e-2 & --   & 7.77555e-2 & -- & 3.0035e-2 & --  \\
\hline
$1/20000$ & 1.53235e-2 & 0.9336 & 4.55894e-2 & 0.7702 & 1.49969e-2 & 1.0020   \\
\hline
$1/40000$ & 7.35736e-3 & 1.0167 & 2.43834e-2 & 0.9028 & 7.11474e-3 & 1.0758   \\
\hline
$1/80000$ & 3.37128e-3 & 1.1677 & 1.13765e-2 & 1.0998 & 3.07743e-3 & 1.2091  \\
\hline
\end{tabular}
\label{biao1}
\end{table}
\begin{table}[!htb]
\centering
\caption{Spatial convergence rates for $\phi$, $\mu$ and $r$ with $\Delta t = 1/320000$ at $T = 0.01$.}
\begin{tabular}{c| c c| c c| c c}
\hline
$h$ & $\|e_{\phi}\|_{\infty,1}$ & Rate & $\|e_{\mu}\|_{\infty,1}$ & Rate & $\|e_{r}\|_{\infty}$ & Rate \\
\hline
$1/8$  & 7.58020e-2 & --      & 3.40001e-1 & --     & 3.77706e-5 & --     \\
\hline
$1/16$ & 3.80908e-2 & 0.9928  & 1.65202e-1 & 1.0413 & 1.03031e-5 & 1.8742 \\
\hline
$1/32$ & 1.85565e-2 & 1.0375  & 7.81922e-2 & 1.0791 & 2.51313e-6 & 2.0355 \\
\hline
$1/64$ & 8.29197e-3 & 1.1621  & 3.45991e-2 & 1.1763 & 5.05922e-7 & 2.3125 \\
\hline
\end{tabular}
\label{biao2}
\end{table}

\end{document}